\documentclass[11pt,onecolumn]{article}
\usepackage{amsmath,amsthm,amsfonts}
\usepackage[bottom=1.3in]{geometry}             
\usepackage[numbers,sort]{natbib}
\usepackage{authblk}
\usepackage{graphicx}
\usepackage{fancyhdr}
\usepackage{amssymb}
\usepackage{amsmath}
\usepackage{epstopdf}
\usepackage[all,cmtip]{xy}
\usepackage[pagebackref,colorlinks,
    linkcolor={red!70!black},
    citecolor={green!60!black},
    urlcolor={blue!80!black}]{hyperref}
\usepackage[capitalize]{cleveref}
\usepackage{stmaryrd}
\usepackage{stackengine}
\usepackage{enumitem}
\usepackage{mathtools}
\usepackage{subcaption}
\usepackage{doi}
\usepackage{tikz}
\usepackage[textsize=footnotesize]{todonotes}
\usepackage{mathrsfs}
\usepackage{backref}
\usepackage[toctitles]{titlesec}

\newcommand\barbelow[1]{\stackunder[1.5pt]{$\hspace{2pt}#1\hspace{2pt}$}{\rule{1ex}{0.075ex}}}

\newcommand{\barabove}[1]{\hspace{2pt}\bar{#1}\hspace{2pt}}
\newcommand{\ot}{\hspace{2pt} \check\otimes\hspace{2pt}}
\theoremstyle{plain}
\newtheorem{theorem}{Theorem}[section]

\newtheorem{corollary}[theorem]{Corollary}
\newtheorem{proposition}[theorem]{Proposition}
\theoremstyle{remark}

\theoremstyle{definition}
\newtheorem{definition}[theorem]{Definition}

\usepackage{makecell}

\usepackage{arydshln}

\title{Self-Dual Cone Systems and Tensor Products}
 
 \author{Tim Netzer}
\affil{\small Department of Mathematics, University of Innsbruck, Austria}
\date{\today}                      
\begin{document}
\maketitle

\begin{abstract} We prove the existence of self-dual tensor products for finite-dimensional convex cones and operator systems. This is   a   consequence of a  more general  result: Every cone system, which is contained in its dual, can be enlarged to a self-dual cone system. Using the  setup of cone systems, we  further describe how all functorial tensor products of finite-dimensional cones and operator systems explicitly arise from the minimal and maximal tensor product.
\end{abstract} 

\section{Introduction and Preliminaries}

The algebraic tensor product of vector spaces is a classical and well-studied construction, that appears in almost all areas of pure and applied mathematics. But when the spaces are equipped with more structure, and the tensor product shall receive  a similar structure, the construction can suddenly become  much more complicated. In particular, there is often more than just one possible tensor product. This happens for example when studying normed spaces, operator spaces, ordered spaces or operator systems. Tensor products of Banach spaces have already been studied by Grothendieck in the 1950s, and have been extended in various directions, for example to topological vector spaces and  operator spaces. We refer to \cite{ryan, pis} for very comprehensive treatments. Tensor products of ordered spaces (or of convex cones) have been studied at least since the 1960s, motivated by questions from functional analysis, optimization, and quantum theory, among others. We refer to the very nice memoir \cite{debruyn} for more details and references. The study of tensor products of operator systems was initiated in \cite{kavruk} and has also developed quite a lot since \cite{pauten, ka1, ka2, ka3, fa}.

One common phenomenon in all of these different cases is the existence of a smallest and a largest tensor product $\barbelow\otimes$ and $\barabove\otimes$ (be aware that which one is called small or large depends on the context and author), which are both well-studied. Usually they have additional nice properties, like symmetry, associativity, functoriality... They are also dual to each other, in the sense that when duality distributes over the product, it turns one into the other: $$(C\barbelow\otimes D)^\vee= C^\vee\barabove\otimes D^\vee\quad \mbox{and}\quad  (C\barabove\otimes D)^\vee= C^\vee\barbelow\otimes D^\vee.$$ In theoretical quantum physics, the minimal tensor product of cones of positive semidefinite matrices contains the so-called \emph{separable bipartite states}: $${\rm Psd}_m(\mathbb C)\barbelow\otimes {\rm Psd}_n(\mathbb C)=\left\{ \sum_{i=1}^d A_i\otimes B_i\mid d\geqslant 1, A_i\in{\rm Psd}_m(\mathbb C), B_i\in{\rm Psd}_n(\mathbb C)\right\}.$$ The maximal tensor product contains the  so-called \emph{bipartite block-positive matrices} (which also correspond to positive maps): \begin{align*} {\rm Psd}_m(\mathbb C)\barabove\otimes {\rm Psd}_n(\mathbb C)=\left\{ X\in{\rm Her}_{m}(\mathbb C)\otimes{\rm Her}_n(\mathbb C)\mid (v\otimes w)^*X(v\otimes w)\geqslant 0\  \forall v\in\mathbb C^m, w\in\mathbb C^n\right\}.\end{align*}
Now one crucial observation is the following: the most important convex cone used to describe bipartite quantum systems is \emph{neither of the two}, but a convex cone in between, namely the full cone of positive semidefinite matrices of the corresponding product dimension, containing the  \emph{bipartite states}: $${\rm Psd}_{mn}(\mathbb C)\subseteq {\rm Her}_{mn}(\mathbb C)={\rm Her}_{m}(\mathbb C)\otimes{\rm Her}_n(\mathbb C).$$ This convex cone really lies exactly in the middle of the minimal and maximal tensor product, in the sense that it is self-dual, whereas the minimal and maximal tensor product cones are dual to each other. So if one defines a new tensor product of psd-cones by just setting  $$ {\rm Psd}_m(\mathbb C)\otimes {\rm Psd}_n(\mathbb C)\coloneqq {\rm Psd}_{mn}(\mathbb C),$$ this is a self-dual tensor product in the following sense: $$ \left({\rm Psd}_m(\mathbb C)\otimes {\rm Psd}_n(\mathbb C)\right)^\vee = {\rm Psd}_m(\mathbb C)^\vee\otimes {\rm Psd}_n(\mathbb C)^\vee.$$
\emph{Generalized probabilistic theories} in quantum physics, initiated in \cite{ma, se}, replace the psd-cones  by other cones, to examine which of the properties of psd-cones are actually necessary for the physical theory (see  \cite{aub, au2, pla} for detailed information and modern results). However, the states are then usually taken as elements from the maximal tensor product, most likely since no other nice tensor product seems to be available for general convex cones. In this sense it is not a true generalization of classical quantum physics, which takes the \emph{tensor product in the middle} as its set of states. 

In this paper we prove  that there exists a self-dual functorial tensor product on the category of all proper finite-dimensional convex cones (\Cref{cor:ctp}). The same is  also true for finite-dimensional  abstract operator systems (\Cref{cor:otp}).

It turns out that these results follow from a  theorem  on abstract cone systems, a concept recently introduced in \cite{vn}. As main general result in the present paper we prove that every abstract cone system, which is contained in its dual, can be enlarged to a self-dual cone system (\Cref{thm:dual}).  Since each tensor product of cones and  of operator systems can be understood  as a certain cone system, and the minimal tensor product is contained in its dual, the maximal one, the existence of a self-dual tensor product follows immediately. Furthermore, since each cone system has a so-called \emph{concrete realization}, we also obtain a concrete description of all functorial tensor products of finite-dimensional cones and operator systems. This is our second main result, \Cref{thm:ctp}. Actually, all these  results hold for any categories of cone systems, the category of finite dimensional convex cones and of finite dimensional operator systems are just two specific examples.

Let us only give a very brief introduction to cone systems, wer refer to \cite{vn} for details, proofs  and many more examples. The notion  was developed to generalize methods and results from the theory of operator systems, and the present paper is an additional justification for doing so. 

We first need a $*$-autonomous category $\tt C$ and a $*$-autonomous functor $\mathcal S\colon {\tt C}\to {\tt FVec}$, called a \emph{stem}. Here, ${\tt FVec}$ denotes the category of finite-dimensional real vector spaces and linear maps. Note that $*$-autonomy is not important to understand the concepts, but it is the crucial property that makes most of the classical results from operator systems theory hold.  Now given some $X\in{\tt FVec}$, an \emph{$\mathcal S$-system on $X$} is a functor $\mathcal G\colon {\tt C}\to {\tt FCone}$, where ${\tt FCone} $ is the category of finite-dimensional closed convex cones and positive maps, such that the following diagram commutes:
$$ \xymatrix{{\tt FVec} \ar[r]^{X\otimes\cdot} & {\tt FVec}\\ {\tt C} \ar[u]^{\mathcal S} \ar[r]^{\mathcal G}& {\tt FCone}\ar[u]_{\mathcal F}}$$ Here $\mathcal F$ is the functor that maps each cone to its underlying space. In words, for every object $c\in {\tt C}$, which should be seen as  the label for  a \emph{level} of the system, $\mathcal G$ assigns a closed convex cone $\mathcal G(c)$ in the space $X\otimes\mathcal S(c)$. These different cones are connected via maps coming from ${\tt C}$, i.e. if $\varphi\in {\tt C}(c,d)$ is a morphism in ${\tt C}$ from $c$ to $d$, then the linear map $${\rm id}_X\otimes\mathcal S(\varphi)\colon X\otimes\mathcal S(c)\to X\otimes \mathcal S(d)$$ is positive, i.e.\ maps $\mathcal G(c)$ into $\mathcal G(d)$.   We will sometimes write $\mathcal G\subseteq\mathcal H$ for $\mathcal S$-systems on $X$, and mean that inclusion holds level-wise, i.e.\ $\mathcal G(c)\subseteq\mathcal H(c)$ for all $c\in{\tt C}$.

When choosing  as ${\tt C}$  the category  ${\tt FHilb}^{\rm op}$ of finite-dimensional Hilbert spaces (opposite to make the following functor covariant),  and the \emph{operator stem } \begin{align*} \mathcal O\colon {\tt FHilb}^{\rm op} & \to {\tt FVec} \\ H & \mapsto \mathbb B(H)_{\rm her} \\ V & \mapsto V^*\cdot V
\end{align*} one recovers the concept  of  abstract operator systems on finite-dimensional spaces. When choosing ${\tt C}={\tt SCone},$ the category of finite-dimensional simplex cones, and as stem  the functor $\mathcal F$ that maps each cone to its underlying space, an $\mathcal F$-system $\mathcal G$ is uniquely defined by the first level cone $\mathcal G(\mathbb R_{\geqslant 0})$, we thus recover the concept of finite-dimensional closed convex cones. Below we will see several more examples of stems.

For a given stem, each of the spaces $\mathcal S(c)$ comes equipped with an \emph{intrinsic} and a \emph{cointrinsic cone}, defined as 
\begin{align*}
A(c)& \coloneqq \{ v\in \mathcal S(c)\mid \mathcal S(\varphi)(v)\geqslant 0 \ \forall \varphi\in {\tt C}(c, \mathbf 1^*)\} \\ B(c) & \coloneqq \overline{\rm cone}\{ \mathcal S(\Phi)(1)\mid \Phi\in {\tt C}(\mathbf 1,c)\}.
\end{align*}
Quite often these two cones coincide, for the operator stem we for example obtain the cone of positive semidefinite operators $A(H)=B(H)={\rm Psd}(H)$ for each $H\in{\tt FHilb}^{\rm op}$.
We further define  the  so-called \emph{intrinsic} and \emph{cointrinsic} $\mathcal S$-systems on the space $\mathcal S(d)$  as $$\mathcal A_d(c)\coloneqq A(d\otimes c), \ \mathcal B_d(c)\coloneqq B(d\ot c).$$ For the operator stem, $\mathcal A_H=\mathcal B_H$ is precisely the usual operator system structure  on the space ${\mathbb B}(H)_{\rm her}$ of Hermitian operators.  

If $\mathcal G,\mathcal H$ are $\mathcal S$-systems on $X$ and $Y$ respectively, a linear map $\psi\colon X\to Y$ is called \emph{completely positive} w.r.t.\ $\mathcal G,\mathcal H$, if $\psi\otimes {\rm id}_{\mathcal S(c)}$ maps $\mathcal G(c)$ into $\mathcal H(c)$, for each $c\in {\tt C}$. 
If $\psi\colon X\to \mathcal S(d)$ is a linear map, the largest system on $X$ that makes $\psi$ completely positive into $\mathcal A_d$ is said to be \emph{concretely realized by $\psi$}. Every $\mathcal S$-system is a (maybe infinite) intersection of systems with a concrete realization.  Every $\mathcal S$-system also has a \emph{dual}, defined as the level-wise dual: $\mathcal G^\vee(c)\coloneqq \mathcal G(c^*)^\vee.$
Finitely many stems $\mathcal S_i\colon{\tt C}_i\to {\tt FVec}$ can be  be combined to a new stem:\begin{align*}\mathcal S_1\otimes\cdots\otimes\mathcal S_d\colon {\tt C}_1\times\cdots\times{\tt C}_d& \to {\tt FVec} \\  (c_1,\ldots, c_d)&\mapsto \mathcal S_1(c_1)\otimes\cdots \mathcal  \otimes S_d(c_d).\end{align*}
For fixed a stem $\mathcal S$ with $A(c)\barbelow\otimes A(d)\subseteq A(c\otimes d)$ for all $c,d\in{\tt C}$, the category ${\tt Sys}(\mathcal S)$ of abstract $\mathcal S$-systems, with completely positive maps, duality  and the minimal tensor product, is again $*$-autonomous, and the functor $\mathcal S'\colon {\tt Sys}(\mathcal S)\to {\tt FVec}$ that maps each system $(\mathcal G,X)$ to its underlying space $X$, is again a stem. 

\section{Self-Dual Cone Systems}\label{sec:sd}

In this section we prove our first main result, \Cref{thm:dual} below. Throughout the rest of this paper, let always $\mathcal S\colon{\tt C}\to {\tt FVec}$  be a stem such that all intrinsic cones $A(c)$ are proper (i.e.\ closed, sharp and with nonempty interior), which is equivalent to all $B(c)$ being proper  \cite[Lemma 3.5]{vn}.
Let $X$ be a fixed finite-dimensional real vector space. We identify $X$ with its dual space $X'$, by choosing a basis $\underline x=(x_1,\ldots, x_d)$ of $X$, denoting the dual basis of $X'$ by $\underline x'=(x_1',\ldots, x_d')$, and using the isomorphism that maps $x_i$ to $x_i'$. This is the same as taking the inner product on $X$ that makes $\underline x$ an orthonormal basis, and using this inner product to identify $X$ and $X'$ in the usual way. However, for ease of notation, we will suppress the isomorphism and just assume $X=X'$. This allows us to compare an $\mathcal S$-system $\mathcal G$ on $X$ with its  dual system $\mathcal G^\vee$ on $X',$ and call such a system \emph{self-dual} if $\mathcal G=\mathcal G^\vee$. However, as often with questions of duality, be aware that this is not basis independent.

In  \cite{ba} it was proven that each finite-dimensional convex cone, which is contained in its dual, can be enlarged to a self-dual cone. When trying to repeat this proof  for operator systems, one already encounters a significant problem. Whereas for $(r_1,\ldots, r_d)\in\mathbb R^d$ we always have $r_1^2+\cdots +r_d^2\geqslant 0$, for $(M_1,\ldots, M_d)\in {\rm Her}_n(\mathbb C)$, the matrix $$M_1\otimes M_1+\cdots + M_d\otimes M_d\in {\rm Her}_{n^2}(\mathbb C)$$ will in general not be positive semidefinite. But this would be needed for a verbatim proof as in \cite{ba}. To deal with this problem,  for $c\in \tt C$ we consider the  set $$\mathcal P(c)\coloneqq \left\{ x\in X\otimes \mathcal S(c)\mid x\in\mathcal W_x^\vee(c)\right\}.$$ Here, $\mathcal W_x$ denotes the $\mathcal S$-system on $X$ generated by $x$, i.e. the smallest system containing $x$. So $x$ belongs to $\mathcal P(c)$ if it is in some sense \emph{positive on itself}. To make this more precise, write  $x=\sum_{i=1}^d x_i\otimes v_i\in X\otimes\mathcal S(c),$ and recall that by \cite[Theorem 3.17]{vn} the dual system $\mathcal W_x^\vee$ is concretely realized  on $\mathcal A_c$ in terms of the linear map \begin{equation}\label{eq:map}x\colon X'=X\to \mathcal S(c);\ x_i\mapsto v_i.\end{equation} So $x\in\mathcal P(c)$ means $$(x\otimes{\rm id}_{\mathcal S(c)})(x)=\sum_{i=1}^dv_i\otimes v_i\in \mathcal A_c(c)= A(c\otimes c),$$ which shows the connection to the above mentioned positivity problem. It also shows that $\mathcal P(c)$ is closed.  The arising collection  $\mathcal P=\left( \mathcal P(c)\right)_{c\in\tt C}$ is closed under linear scaling and applying morphisms from $\tt C$ to the second factor, but not convex and thus  not an $\mathcal S$-system in general.
The following proposition contains the most important facts about $\mathcal P$ needed in the proof of our  main theorem below.

\begin{proposition}\label{prop:p}
(i) If $\mathcal G$ is an $\mathcal S$-system on $X$ with $\mathcal G\subseteq \mathcal G^\vee,$ then $\mathcal G\subseteq\mathcal P$.

\noindent
(ii) For $x\in\mathcal P(c)$ we have $\mathcal W_x\subseteq \mathcal P$.

\noindent
(iii) If $0\neq x\in\partial \mathcal P(c)$, then $x\in\partial \mathcal W_x^\vee(c).$
\end{proposition}
\begin{proof}
($i$) For $x\in\mathcal G(c)$ we have $\mathcal W_x\subseteq \mathcal G$, so $$x\in \mathcal W_x(c)\subseteq\mathcal G(c) \subseteq \mathcal G^\vee(c)\subseteq \mathcal W_x^\vee(c)$$ and thus $x\in\mathcal P(c)$. For ($ii$) let $x\in\mathcal P(c)$. From $x\in\mathcal W_x^\vee(c)$ we obtain $\mathcal W_x\subseteq \mathcal W_x^\vee,$ and ($i$) thus implies $\mathcal W_x\subseteq \mathcal P$. 

For ($iii$) write $x=\sum_{i=1}^d x_i\otimes v_i\in X\otimes\mathcal S(c)$. In the following we will freely switch between the use of $x$ as an element of $X\otimes \mathcal S(c)$ and as a linear map $X=X'\to \mathcal S(c)$ as in (\ref{eq:map}).

First assume that the span of the $v_i$ (i.e. the image of the map $x$) intersects the interior of the intrinsic cone $A(c)$. Then all maps $x\otimes{\rm id}_{\mathcal S(d)}$ hit the interior of $\mathcal A_c(d)=A(c\otimes d)$, since the cointrinsic cones $B(d)$ all have nonempty interior and ${\rm int}(A(c))\barbelow\otimes{\rm int}(B(d))\subseteq {\rm int} (A(c\otimes d))$ holds,  see \cite[Lemma 3.8]{vn}. Since $\mathcal W_x^\vee$ is concretely realized by the map $x$ \cite[Theorem 3.17]{vn}, the boundary of $\mathcal W_x^\vee(c)$ is the inverse image of the boundary of $\mathcal A_c(c)=A(c\otimes c)$. From $x\in\partial \mathcal P(c)$ we clearly obtain $\sum_i v_i\otimes v_i\in \partial A(c\otimes c)$, which thus implies  $x\in\partial \mathcal W_x^\vee(c)$.

Second assume that the space spanned by the $v_i$ intersects $A(c)$ only at the  boundary (including the case where it only intersects at the origin). By definition of $A(c)$ as the dual of all $\mathcal S(\varphi)$ for $\varphi\in\tt C(c,\bf 1^*),$  we find some $\varphi\in \tt  C(c,\mathbf 1^*)$ such that  $\mathcal S(\varphi)$  vanishes at all $v_i,$ but is positive on the relative interior of $A(c)$ (strictly speaking one might need to use a positive combination of such functionals, but that doesn't change the argument). Choose some $w\in\mathcal S(c)$ with $\mathcal S(\varphi)(w)\neq 0$. Further, since $x\neq 0$, at least one $v_i\neq 0$, so assume w.l.o.g.\ that $v_1\neq 0$. Now choose  some $\psi\in{\tt C}(c,\bf 1^*)$ with $\mathcal S(\psi)(v_1)\neq 0$ (if there was no such $\psi$, $v_1$ would belong to $A(c)\cap -A(c)=\{0\}$). Then $$\left(\underbrace{\mathcal S(\psi)\otimes \mathcal S(\varphi)}_{=\mathcal S(\psi\otimes\varphi)}\right)\left(\sum_i v_i\otimes (v_i\pm \varepsilon \delta_{i1}w)\right)=\pm\varepsilon \underbrace{\mathcal S(\psi)(v_1)}_{\neq 0}\underbrace{\mathcal S(\varphi)(w)}_{\neq 0}.$$ For one of the choices $\pm$ this is negative, which shows that one of $\sum_i v_i\otimes (v_i\pm\varepsilon\delta_{i1} w)\notin A(c\otimes c),$ and thus $x_\varepsilon\coloneqq\sum_i x_i\otimes (v_i\pm\varepsilon\delta_{i1} w)\notin \mathcal W_x^\vee(c).$ Since $x_\varepsilon$ converges to $x$ for $\varepsilon\to 0$, this proves  $x\in\partial W_x^\vee(c).$
\end{proof}

We now state and prove our first main result.  As mentioned before, \Cref{thm:dual} for convex cones alone is much older \cite[Theorem 5]{ba}. Our proof starts similarly, but we will encounter  some additional difficulties for which we need \cref{prop:p}.
\begin{theorem} \label{thm:dual}
Let  $\mathcal G$ be an $\mathcal S$-system on $X$  with  $\mathcal G\subseteq \mathcal G^{\vee}$. Then there exists an $\mathcal S$-system $\mathcal E$ on $X$  with $$\mathcal G\subseteq \mathcal E=\mathcal E^\vee\subseteq \mathcal G^\vee.$$
\end{theorem}
\begin{proof}
We consider the nonempty set $$\mathcal M\coloneqq\left\{(\mathcal H,X)\in{\tt Sys}(\mathcal S)\mid \mathcal G\subseteq\mathcal H\subseteq\mathcal H^\vee \right\},$$ partially ordered by inclusion. Let us check that the assumption of Zorn's Lemma is fulfilled. So let $\left(\mathcal H_i\right)_{i\in I}$ be a chain in $\mathcal M$. For any $i,j\in I$ we either have $\mathcal H_i\subseteq \mathcal H_j\subseteq \mathcal H_j^\vee,$  or $\mathcal H_j\subseteq \mathcal H_i,$ in which case $\mathcal H_i\subseteq \mathcal H_i^\vee\subseteq \mathcal H_j^\vee$ as well. So $$\sum_i \mathcal H_i \subseteq \bigcap_i \mathcal H_i^\vee=\left(\sum_i \mathcal H_i\right)^\vee,$$ which shows that $\sum_i\mathcal H_i$ is an upper bound for the chain in $\mathcal M$. So Zorn's Lemma guarantees the existence of a maximal element $\mathcal E\in\mathcal M.$

We first show that $\mathcal E=\mathcal E^\vee\cap \mathcal P$, where the inclusion from left to right is clear from \Cref{prop:p}($i$). So take $x\in\mathcal E^\vee(c)\cap\mathcal P(c)$. We consider the $\mathcal S$-system $$\widetilde{\mathcal E}\coloneqq \mathcal E+ \mathcal W_x$$ generated by $\mathcal E$ and $x,$ and claim that $\widetilde{\mathcal E}\subseteq \widetilde{\mathcal E}^\vee$ holds. By maximality of $\mathcal E$ this will imply $\mathcal E=\widetilde{\mathcal E}$ and thus $x\in\mathcal E(c)$. 

From $x\in \mathcal E^\vee(c)$ we obtain $\mathcal W_x\subseteq\mathcal E^\vee$, and biduality \cite[Theorem 3.17]{vn} implies $\mathcal E\subseteq \mathcal W_x^\vee.$
We thus get $${\widetilde{\mathcal E}}^\vee=\mathcal E^\vee\cap \mathcal W_x^\vee\supseteq \mathcal E\cap \mathcal W_x^\vee=\mathcal E.$$  But since $x\in\mathcal P(c)$ we also have $x\in\mathcal W_x^\vee(c),$  so $x\in\widetilde{\mathcal E}^\vee(c)$ and thus $\widetilde{\mathcal E}\subseteq \widetilde{\mathcal E}^\vee.$ 

So we have indeed shown $\mathcal E=\mathcal E^\vee\cap \mathcal P.$
We now argue that this already implies $\mathcal E=\mathcal E^\vee$. Take $0\neq x\in\mathcal E(c)\cap \partial \mathcal P(c).$ Then $\mathcal W_x\subseteq\mathcal E$ and thus $\mathcal E^\vee\subseteq \mathcal W_x^\vee$. By \Cref{prop:p}($iii$)  we also  have $x\in\partial \mathcal W_x^\vee(c)$, which then clearly also implies $x\in\partial \mathcal E^\vee(c)$. This finally shows $\mathcal E=\mathcal E^\vee$.
\end{proof}

When applied to the simplex stem, \Cref{thm:dual} recovers the result from \cite{ba} that each cone contained in its dual can be extended to a self-dual cone.
When applied to the operator stem, \Cref{thm:dual} reads as follows: 
\begin{corollary}
Let $\mathcal G$ be an abstract operator system on the finite-dimensional space $X$. If $\mathcal G\subseteq \mathcal G^\vee,$ then $\mathcal G$ can be enlarged to a self-dual operator system on $X$.
\end{corollary}

\section{Self-Dual Tensor Products of Cone Systems}

In this section we apply \Cref{thm:dual} to prove the existence of several self-dual functorial tensor products. Again let $\mathcal S\colon\tt C\to\tt{FVec}$ be a stem for which all intrinsic cones are proper. We further assume that$$A(c)\barbelow\otimes  A(d)\subseteq A(c\otimes d)$$ holds for all $c,d\in{\tt C}$, which is needed for tensor products of $\mathcal S$-systems to exist \cite[Section 3.5]{vn}. Recall that  $\left({\tt Sys}(\mathcal S),\barbelow\otimes,\mathcal B_{\bf 1^*},\vee\right)$ is a $*$-autonomous category
and $$\mathcal S'\colon{\tt Sys}(\mathcal S)\to {\tt FVec};\ (\mathcal G,X)\mapsto X$$ is again a stem. 
In this section we will  restrict ourselves to ${\tt Sys}_{\rm pr}(\mathcal S),$ the full subcategory of $\mathcal S$-systems with full-dimensional cones at level $\bf 1$ and sharp cones at level  $\bf 1^*$ (this implies that all cones at all levels are proper \cite[Lemma 3.8]{vn}). The restriction of  $\mathcal S'$ to ${\tt Sys}_{\rm pr}(\mathcal S)$ is clearly also a stem. 

For the operator stem $\mathcal O$, the category ${\tt Sys}_{\rm pr}(\mathcal O)$ is the familiar category of {abstract operator systems  on finite-dimensional spaces (the properness of cones is part of the assumption for abstract operator systems), that we also denote by  {\tt FAOS}. For the simplex stem $\mathcal S$, ${\tt Sys}_{\rm pr}(\mathcal S)$ equals the category ${\tt PCone}$ of proper finite-dimensional cones.

\begin{definition}\label{def:tp}
A \emph{tensor product of proper $\mathcal S$-systems}  is a functor $$\otimes\colon {\tt Sys}_{\rm pr}(\mathcal S)\times{\tt Sys}_{\rm pr}(\mathcal S) \to {\tt Sys}_{\rm pr}(\mathcal S)$$ for which the following diagram commutes
$$\xymatrix{{\tt Sys}_{\rm pr}(\mathcal S)\times{\tt Sys}_{\rm pr}(\mathcal S)\ar[rr]^{\qquad\otimes}\ar[dr]_{\mathcal S'\otimes\mathcal S'}&&{\tt Sys}_{\rm pr}(\mathcal S) \ar[dl]^{\mathcal S'}\\ &{\tt FVec} &}$$ 
In words, to any two $\mathcal S$-systems $(\mathcal G,X),(\mathcal H,Y)\in{\tt Sys}_{\rm pr}(\mathcal S)$ a tensor product assigns a proper $\mathcal S$-system on the space $X\otimes Y$, in a way that is functorial with respect to completely positive maps. 
\end{definition}

Note that this notion slightly differs from the one given in \cite{vn, kavruk}, but the smallest and largest tensor products $\barbelow\otimes, \barabove\otimes$ from \cite{vn, kavruk}  are tensor products in our sense (see \Cref{sec:tens} for a more thorough discussion of the differences).
For each tensor product $\otimes$ on ${\tt Sys}_{\rm pr}(\mathcal S)$ there exists the \emph{dual tensor product} $\otimes^\vee,$ defined as $$\mathcal G\otimes^\vee \mathcal H\coloneqq \left(\mathcal G^\vee\otimes\mathcal H^\vee\right)^\vee.$$ The dual of the smallest is the largest tensor product, $\barbelow{\otimes}^\vee=\barabove\otimes,$ and vice versa. A tensor product is \emph{self-dual} if $\otimes=\otimes^\vee.$ This means that $$\left(\mathcal G\otimes\mathcal H\right)^\vee=\mathcal G^\vee\otimes\mathcal H^\vee$$ holds for all $\mathcal G,\mathcal H$, i.e. $\otimes$ is compatible with duality.
We now prove that in between $\barbelow\otimes$ and $\barabove\otimes$ there  exists a self-dual tensor product on ${\tt Sys}_{\rm pr}(\mathcal S)$. 

\begin{theorem}\label{thm:tens} On ${\tt Sys}_{\rm pr}(\mathcal S)$ there exists a self-dual tensor product in between $\barbelow\otimes$ and $\barabove\otimes.$
\end{theorem}
\begin{proof}
Consider the threefold stem \begin{align*}\widetilde{\mathcal S}\coloneqq\mathcal S'\otimes\mathcal S'\otimes\mathcal S\colon{\tt Sys}_{\rm pr}(\mathcal S)\times{\tt Sys}_{\rm pr}(\mathcal S)\times \tt C &\to {\tt FVec} \\ \left((\mathcal G,X),(\mathcal H,Y), c\right) &\mapsto X\otimes Y\otimes\mathcal S(c).\end{align*} The arising (co)intrinsic cones are minimal/maximal tensor products of the single (co)intrinsic cones \cite[Section 4.2]{vn}.  Our properness assumptions  thus imply that all intrinsic cones of $\widetilde{\mathcal S}$ are proper, and we can apply \Cref{thm:dual}.
An $\widetilde{\mathcal S}$-system on the space $\mathbb R$  is a functor $\mathcal T$ such that the following diagram commutes: 
$$\xymatrix{{\tt FVec} \ar[r]^{{\rm id}=\mathbb R\otimes\cdot}& {\tt FVec}\\ {\tt Sys}_{\rm pr}(\mathcal S)\times{\tt Sys}_{\rm pr}(\mathcal S)\times \tt C \ar[u]_{\widetilde{\mathcal S}}\ar[r]^{\quad\qquad\mathcal T} & {\tt FCone}\ar[u]^{\mathcal F}}$$ Assuming that such  $\mathcal T$ assigns only proper cones, we define for $\mathcal G,\mathcal H\in{\tt Sys}_{\rm pr}(\mathcal S)$ $$(\mathcal G\otimes\mathcal H)(c)\coloneqq \mathcal T(\mathcal G,\mathcal H,c)$$
and see that this defines a one-to-one correspondence between $\widetilde{\mathcal S}$-systems on $\mathbb R$ (assigning only proper cones) and tensor products on ${\tt Sys}_{\rm pr}(\mathcal S).$ 
This correspondence is also  compatible with the notion of duality.

Now we consider the minimal tensor product  on ${\tt Sys}_{\rm pr}(\mathcal S)$ \cite[Section 3.5]{vn} as an $\widetilde{\mathcal S}$-system $\mathcal T$. Since the dual of the minimal tensor product is the maximal tensor product, which contains the minimal one, we have $\mathcal T\subseteq\mathcal T^\vee.$
From \Cref{thm:dual}  we obtain a self-dual $\widetilde{\mathcal S}$-system in between $\mathcal T$ and $\mathcal T^\vee$. This clearly only assigns proper cones, and thus  gives rise to a self-dual tensor product on ${\tt Sys}_{\rm pr}(\mathcal S)$, in between $\barbelow\otimes$ and $\barabove\otimes$.
\end{proof}

\begin{corollary}\label{cor:ctp} In between the minimal and maximal tensor product, there exists a self-dual functorial tensor product on the category ${\tt PCone}$ of finite-dimensional proper cones. 
\end{corollary}
\begin{proof} For the simplex stem $\mathcal S$ we obtain ${\tt Sys}_{\rm pr}(\mathcal S)={\tt PCone},$ so the statement follows from \Cref{thm:tens}.
\end{proof}

\begin{corollary}\label{cor:otp}
In between the minimal and maximal tensor product, there exists a self-dual functorial tensor product on the category ${\tt FAOS}$ of finite-dimensional abstract operator systems.
\end{corollary}
\begin{proof}
For the operator stem $\mathcal O$ we obtain ${\tt Sys}_{\rm pr}(\mathcal O)={\tt FAOS},$ so the statement follows from \Cref{thm:tens}.
\end{proof}

\section{Characterizing All Tensor Products of Cone Systems}\label{sec:tens}

Tensor products of proper $\mathcal S$-systems (in the sense of \Cref{def:tp}) correspond to  $\widetilde S$-systems on $\mathbb R,$ that assign only proper cones. From the assumption that all intrinsic cones are proper, this properness is equivalent to the cone at base level being $\mathbb R_{\geqslant 0}$ or $\mathbb R_{\leqslant 0}$. However, the notion of tensor products from \cite{vn, kavruk} is slightly different. First, tensor products from \cite{vn,kavruk} are not required to be functorial w.r.t.\ completely positive maps in general, and explicitly called \emph{functorial} if they are. Since functoriality is a very important and  useful property, and most known tensor products of cones and operator systems are actually functorial, it is not a big deficit that our approach only covers functorial tensor products. Second, tensor products in the sense of \cite{vn,kavruk} are supposed to contain all minimal tensor products of the pairwise level cones of the systems. This is not automatic for tensor products in the sense of \Cref{def:tp}, or equivalently for $\widetilde{\mathcal S}$-systems on $\mathbb R.$ First, it clearly requires the $\widetilde{\mathcal S}$-systems to have $\mathbb R_{\geqslant 0}$ at base level. But even this assumption (which is very reasonable)  is not enough in general. For example, the minimal $\widetilde{\mathcal S}$-system $\underline{\mathcal G}_{\mathbb R_{\geqslant 0}}$ over $\mathbb R_{\geqslant 0}$  gives rise to the following tensor product: $$\mathcal G\otimes\mathcal H\coloneqq \underline{\mathcal G}_{\mathcal G(1)\barbelow\otimes \mathcal H(1)},$$ i.e.\ the minimal system over the minimal tensor product over the respective base cones. This is indeed a tensor product in the sense of our \Cref{def:tp}, but not one in the sense of  \cite{vn, kavruk}, since it is too small in general.
Third, tensor products in the sense of \cite{vn,kavruk} have to fulfill a certain quasi-functoriality, i.e.\ cp maps to intrinsic systems have to tensor to  cp maps  into the corresponding  product intrinsic system. In our setup, the maximal $\widetilde{\mathcal S}$-system $\overline{\mathcal G}_{\mathbb R_{\geqslant 0}}$ over $\mathbb R_{\geqslant 0}$ gives rise to the following tensor product: $$\mathcal G\otimes\mathcal H\coloneqq \overline{\mathcal G}_{\mathcal G(\mathbf 1^*)\barabove\otimes\mathcal H(\mathbf 1^*)},$$ which is larger than the maximal tensor product in the sense of \cite{vn,kavruk} in general. 
Summarizing, the functorial tensor products in the sense of 
\cite{vn,kavruk} are precisely those tensor products in the sense of \Cref{def:tp} (or those  $\widetilde{\mathcal S}$-systems over $\mathbb R_{\geqslant 0}$) that lie in between $\barbelow\otimes$ and $\barabove\otimes.$

The following result describes how all tensor products arise from the minimal and maximal tensor product.
\begin{theorem}\label{thm:ctp}
All tensor products of proper $\mathcal S$-systems (in the sense of \Cref{def:tp}) arise in the following way: Take an index set $I$, and  for every $i\in I$ choose  $(\mathcal D_i,V_i),(\mathcal E_i,W_i)\in{\tt Sys}_{\rm pr}(\mathcal S), c_i\in {\tt C}$ and $x_i\in V_i\otimes W_i\otimes\mathcal S(c_i).$ Then define for any $\mathcal G,\mathcal H\in {\tt Sys}_{\rm pr}(\mathcal S)$ and $c\in{\tt C}$ $$a\in (\mathcal G\otimes\mathcal H)(c)\Leftrightarrow \forall i \colon\ x_i\otimes a \in (\mathcal D_i\barbelow\otimes \mathcal G)(\mathbf 1^*)\barabove\otimes (\mathcal E_i\barbelow\otimes \mathcal H)(\mathbf 1^*)\barabove\otimes  A(c_i\otimes c).$$
The  tensor product contains $\barbelow\otimes$ if  all $x_i\in(\mathcal D_i\barabove\otimes\mathcal E_i)(c_i).$ The  tensor products that are additionally  contained in $\barabove\otimes$, i.e.\  functorial tensor products in the sense of \cite{vn, kavruk},   are clearly obtained by intersecting with $\barabove\otimes$.
\end{theorem}
 \begin{proof} We have seen that the tensor products  in the sense of \Cref{def:tp} are in one-to-one correspondence with $\widetilde{\mathcal S}$-systems $\mathcal T$  over the base cones $\mathbb R_{\geqslant 0}$ or $\mathbb R_{\leqslant 0}$. We now  apply the realization theorem \cite[Theorem 3.22]{vn} to these systems. Indeed, any concretely realized system on $\mathbb R$ is induced by  a linear map $$\psi_i\colon \mathbb R\to \widetilde{\mathcal S}(\mathcal D_i,\mathcal E_i,c_i)=V_i\otimes W_i\otimes\mathcal S(c_i).$$ With $x_i\coloneqq \psi_i(1)$ and the observation that the intrinsic $\widetilde{\mathcal S}$-cones are precisely the ones on the right in the statement of the theorem, this proves the claim.  (To be precise, to actually obtain a system over $\mathbb R_{\geqslant 0}$ or $\mathbb R_{\leqslant 0}$, i.e.\ a tensor product of proper systems, the $x_i$ all have to be chosen from the intrinsic cones $\pm\ {\mathcal D}_i(\mathbf 1^*)\barabove\otimes{\mathcal E}_i(\mathbf 1^*)\barabove\otimes A(c_i)$.)

 An $\widetilde{\mathcal S}$-system $\mathcal T$ contains $\barbelow\otimes$   if and only if ${\mathcal T}^\vee \subseteq \barabove\otimes.$ If $\mathcal T$ is contained in a system with a concrete realization, the system generated by the realizing map is contained in the dual of $\mathcal T$ and thus in $\barabove\otimes$. In our case, it is easily checked that this means $x_i\in (\mathcal D_i\barabove\otimes\mathcal E_i)(c_i)$ for all $i$.
\end{proof}

 \begin{corollary}
 For each functorial tensor product $\otimes$ of proper cones with $\mathbb R_{\geqslant 0}\otimes\mathbb R_{\geqslant 0}=\mathbb R_{\geqslant 0}$ there exist  proper cones  $(D_i,V_i),(E_i,W_i)\in{\tt PCone}$  and $x_i\in D_i\barabove\otimes E_i$ such that  for all $(G,X),(H,Y)\in{\tt PCone}$ we have $$G\otimes H=\left\{a\in X\otimes Y\mid x_i\otimes a \in (D_i\barbelow\otimes G)\barabove\otimes(E_i\barbelow\otimes H) \mbox{ for all } i  \right\}.$$ 
\end{corollary}
\begin{proof}
Apply \Cref{thm:ctp} to proper simplex systems, where all systems are determined by their cones at base level. 
Here the minimal $\widetilde{\mathcal S}$-system over $\mathbb R_{\geqslant 0}$ really corresponds to the minimal tensor product.
\end{proof}

\begin{corollary}
Let $\otimes$ be a tensor product of finite dimensional abstract operator systems (in the sense of \Cref{def:tp}). Then there exist finite dimensional abstract operator systems  $(\mathcal D_i,V_i),(\mathcal E_i,W_i)\in{\tt FAOS}, n_i\in \mathbb N$ and $x_i\in \pm\mathcal D_i(1)\barabove \otimes \mathcal E_i(1)\barabove\otimes {\rm Psd}_{n_i}(\mathbb C),$ such that for all $\mathcal G,\mathcal H\in {\tt FAOS}, n\in\mathbb N$ we have $$a\in (\mathcal G\otimes\mathcal H)(n)\Leftrightarrow \forall i \colon\ x_i\otimes a \in (\mathcal D_i\barbelow\otimes \mathcal G)(1)\barabove\otimes (\mathcal E_i\barbelow\otimes \mathcal H)(1)\barabove\otimes  {\rm Psd}_{n_i n}(\mathbb C).$$ The  tensor product contains $\barbelow\otimes$ if all $x_i\in(\mathcal D_i\barabove\otimes\mathcal E_i)(n_i).$
\end{corollary}
\begin{proof} This is just \Cref{thm:ctp} for the operator stem.
\end{proof}

\bibliographystyle{abbrvnat}
\bibliography{references}
\end{document}